\documentclass{amsart}

\newtheorem{theorem}{Theorem}[section]
\newtheorem{lemma}[theorem]{Lemma}
\newtheorem{proposition}[theorem]{Proposition}
\newtheorem{corollary}[theorem]{Corollary}

\theoremstyle{definition}
\newtheorem{definition}[theorem]{Definition}
\newtheorem{example}[theorem]{Example}

\theoremstyle{remark}
\newtheorem{remark}[theorem]{Remark}
\newtheorem{acknowledgment}[theorem]{Acknowledgment}

\numberwithin{equation}{section}

\begin{document}
\title{Stampacchia's property, self-duality and orthogonality relations}
\author{Nikos Yannakakis}
\address{Department of Mathematics\\
National Technical University of Athens\\
Iroon Polytexneiou 9\\
15780 Zografou\\
Greece}
\email{nyian@math.ntua.gr}
\date{}
\begin{abstract}
We show that  if the conclusion of the well known Stampacchia Theorem, on variational inequalities, holds on a Banach space $X$, then $X$ is isomorphic to a Hilbert space. Motivated by this we obtain a relevant result concerning self-dual Banach spaces and investigate some connections between existing notions of orthogonality and self-duality. Moreover, we revisit the notion of the cosine of a linear operator and show that it can be used to characterize Hilbert space structure. Finally, we present some consequences of our results to quadratic forms and to evolution triples.
\end{abstract}
\maketitle
\section{Introduction}
Let $H$ be a real Hilbert space, $\|\cdot\|$ be its norm, $(\cdot\,,\cdot)$ its inner product and let $$a:H\times H\rightarrow \mathbb R$$ 
be a bounded bilinear form.

The well known Stampacchia Theorem (also called the Lions-Stampacchia Theorem, see \cite{Ernst}, \cite{Stampacchia1} and \cite{Stampacchia2}) states that if the above bilinear form is {\it coercive}, i.e. there exists $c>0$ such that
\begin{equation}
\label{coercive}
a(x,x)\geq c\| x\|^2,\text{ for all } x\in H,
\end{equation}
then for any nonempty, closed, convex subset $M$ of $H$ and $h\in H$, there exists a unique solution $x\in M$, of the variational inequality
\begin{equation}
\label{var}
a(x,z-x)\geq (h, z-x)\,,\text{ for all }z\in M.
\end{equation}

Our first aim, in this paper, is to investigate whether Stampacchia's Theorem can be generalized in the broader setting of an arbitrary Banach space $X$. As we will see, at least in its full generality, this is impossible since its conclusion implies that $X$ has to be isomorphic to a Hilbert space. 

In the sequel we obtain a relevant result concerning self-dual Banach spaces, i.e. Banach spaces that are isomorphic to their dual spaces. Along the way we see that our approach brings out some connections between existing notions of orthogonality in general normed linear spaces and self-duality. 

In the last section and motivated by the above, we revisit the cosine of a linear operator (a notion originally introduced by K. Gustafson in \cite{gus1}) and use it to obtain an additional Hilbert space characterization based on a result of J. R. Partington which can be found in \cite{Partington}. 

Finally, we present some consequences of our results to quadratic forms and to evolution triples.


\section{Stampacchia's property}
\label{section1}
Let $X$ be a real Banach space, $X^\ast$ be its dual and $\langle \cdot\,,\cdot\rangle$ be their duality product. By $M^\bot$ we denote the annihilator of a subspace $M$ of $X$, i.e.
$$M^\bot=\left\{x^\ast\in X^\ast:\,<x^\ast,x>=0\,, \text{ for all }x\in M\right\}\,.$$

To obtain the natural analogue of the conclusion of Stampacchia's Theorem in this situation we need the following definition.
\begin{definition}
Let $X$ be a real Banach space. We say that $X$ has Stampacchia's property (property (S) for short), if there exists
a bounded, bilinear form
$$a:X\times X\rightarrow\mathbb{R}$$
such that if $M$ is any nonempty, convex, closed subset of $X$ and $x^\ast\in X^\ast$, then there exists
a unique $x\in M$ such that
$$a(x,z-x)\geq \langle x^\ast,z-x\rangle\,,\text{ for all } z\in M.$$

\end{definition}
Recall that a closed subspace $M$ of a Banach space $X$ is complemented in $X$, if there exists another closed subspace $N$ of $X$ such that $X$ is their direct sum, i.e. 
$$M\cap N=\left\{\,0\right\}\text{ and }X=M+N\,.$$
Note that the existence of such a closed subspace $N$ is equivalent to the existence of a bounded linear projection from $X$ onto $M$.

Not all closed subspaces of an arbitrary Banach space are complemented. In fact we have the following well-known result by J. Lindenstrauss and L. Tzafriri \cite{LinTza}, which we will use in the sequel.
\begin{theorem}
\label{Lintza}
A Banach space $X$ is isomorphic to a Hilbert space if and only if all its closed subspaces are complemented.
\end{theorem}

To proceed with our task we need the following simple lemma. 
\begin{lemma}
\label{complemented}
Let $X$ be a Banach space and $M$ be a closed subspace of $X$. If there exists another Banach space $Y$ and a bounded linear operator
$$S:M\rightarrow Y$$
$1-1$ and onto $Y$, which can be extended to the whole of $X$, then the closed subspace $M$ is complemented in $X$.
\end{lemma}
\begin{proof}
If 
$$\hat{S}:X\rightarrow Y$$ 
denotes the extension of $S$ to the whole of $X$, then it is easy to see that the operator 
$$S^{-1}\circ\hat{S}:X\rightarrow M$$
is the required bounded projection onto $M$.
\end{proof}  

We can now show that property (S) characterizes Hilbert space structure.
\begin{theorem}
\label{theorem}
A real Banach space $X$ is isomorphic to a Hilbert space if and only if it has property (S).
\end{theorem}
\begin{proof}
The neccesity is obvious. We prove that property (S) is also sufficient.
To this end let $M$ be any closed subspace of $X$. We will show that $M$ is complemented. Since $M$ is a closed subspace of $X$, it is easy to see that property (S) in particular implies that for all
$x^\ast\in X^\ast$, there exists a unique $x\in M$ such that
\begin{equation}
\label{eq2}
a(x,z)=\langle x^\ast,z\rangle,\text{ for all }z\in M.
\end{equation}
Define the bounded linear operator
$$T:X\rightarrow X^\ast\,,$$ 
by 
$$\langle Tx,z\rangle=a(x,z),\text{ for all }x, z\in X\,$$ 
and let
$$\pi:X^\ast\rightarrow X^\ast/M^\bot$$
be the natural quotient map. 

Then the restriction of the operator $\pi \circ T$ on the subspace $M$ is $1-1$ and onto $X^\ast/M^\bot$. 

To see this first note that if 
$$(\pi \circ T)x=0$$ 
then $Tx\in M^\bot$, i.e. $\langle Tx,z\rangle=0$, for all $z\in M$. By the definition of $T$ this implies that $a(x,z)=0$, for all $z\in M$. But by hypothesis there is a unique $x\in M$ such that $a(x,z)=0$, for all $z\in M$, which by the boundedness of $a$ has to be $0$. Thus the restriction of the operator $\pi \circ T$ on the subspace $M$  is $1-1$.

To show that $\pi \circ T|_M$ is also onto, let $h=x^\ast+M^\bot$, for some $x^\ast\in X^\ast$. Then by (\ref{eq2}) we have that there exists a unique $x\in M$ such that $a(x,z)=\langle x^\ast,z\rangle$, for all $z\in M$. Hence again by the definition of $T$ we have that 
$$\langle Tx,z\rangle=\langle x^\ast,z\rangle\,, \text{ for all }z\in M\,,$$ 
i.e. $Tx-x^\ast\in M^\bot$. Hence $(\pi\circ T)(x)=h$ and thus $\pi \circ T|_M$ is onto $X^\ast/M^\bot$.

Note now that by its definition the operator 
$$S=\pi \circ T|_M$$ 
can be trivially extended to the whole of $X$ and thus by Lemma \ref{complemented} the closed subspace $M$ is a complemented subspace of $X$. Since $M$ was arbitrary we get by Theorem \ref{Lintza} that $X$ is isomorphic to a Hilbert space. 
\end{proof}
\begin{remark}
A careful look in the above proof shows that if the Banach space $X$ has property (S) and $M$ is any closed subspace of $X$ then 
$$X=M\oplus T^{-1}(M^\bot)\,,$$
where $T$ is the operator associated to the bilinear form $a(\cdot,\cdot)$.
\end{remark}

\begin{remark}
\label{remark}
A main hypothesis in Stampacchia's Theorem is the coercivity condition (\ref{coercive}). As it is well-known (see for example \cite{DY}, \cite{Lin}), such a hypothesis cannot hold in an arbitrary Banach space $X$ since if it did, then $X$ would have an equivalent Hilbertian norm induced by the inner product
$$(x,y)=\frac{1}{2}[a(x,y)+a(y,x)]$$
and thus would be isomorphic to  a Hilbert space.
Hence our result implies that there can be no full generalization of Stampacchia's Theorem in an arbitrary Banach space even if one drops the coercivity condition (\ref{coercive}). 
\end{remark}
\begin{remark}
\label{bounded}
Note that if we are restricted to bounded closed and convex subsets of a 
Banach space $X$, then a generalization of Stampacchia's Theorem is possible by just assuming that the bilinear form is strictly positive i.e. 
$$a(x,x)>0\,,\text{ for all }x\in X.$$ 
The proof is a straightforward application of a result due to Brezis \cite[Theorem 24]{Brezis1}, on pseudomonotone operators. It is easy to see that in this case the space $X$ need not be isomorphic to a Hilbert space.
\end{remark}
\begin{remark}
It seems appropriate to mention here a recent result by E. Ernst and M. Th\' era: if as in Remark \ref{bounded} we are restricted to bounded, closed and convex sets and moreover $X$ is a Hilbert space, then the pseudomonotonicity of the operator associated to the bilinear form $a(\cdot,\cdot)$, is a necessary and sufficient condition for the existence of a solution of the variational inequality (\ref{var}) (see \cite[Theorem 3.1]{Ernst}).  
A similar result for unbounded sets has been obtained by A. Maugeri and F. Raciti in \cite{maugeri}.
\end{remark}
\section{Self-dual Banach spaces}
\label{self}
A self-dual Banach space is a Banach space isomorphic to its dual. It is well-known that Hilbert spaces are self-dual although they are far from being the only ones; if $Y$ is any reflexive Banach space then
$$X=Y\oplus Y^\ast$$ 
is self-dual. 

We will now see that our approach in Section \ref{section1} can lead us to a result concerning self-dual Banach spaces.  
The important observation is the fact that the operator $T$ associated to the bilinear form $a(\cdot,\cdot)$, in the proof of Theorem \ref{theorem}, is an isomorphism from $X$ onto $X^\ast$ and hence $X$ is a self-dual space. 

Our result is the following. 
\begin{proposition}
\label{th1}
Let $X$ be a real, self-dual, Banach space. If the isomorphism
$$T:X\rightarrow X^\ast$$
onto $X^\ast$ is such that for any closed subspace $M$ of $X$, the map
$\pi\circ T|_M$  is an isomorphism onto $X^\ast/M^\bot$, where $\pi$ is the natural quotient map from $X^\ast$ onto $X^\ast/M^\bot$, then $X$ is isomorphic to a Hilbert space.
\end{proposition}
\begin{proof}
We follow the proof of Theorem \ref{theorem}.
\end{proof}
\begin{remark}
Recalling that the quotient space $X^\ast/M^\bot$ is isomorphic to $M^\ast$ we can rephrase Proposition \ref{th1} as follows: 

``Let $X$ be a self-dual space. If the isomorphism  between $X$ and $X^\ast$ induces in a \textit{natural way} (through the natural quotient maps) isomorphisms between all closed subspaces of $X$ and their corresponding duals, then $X$ is isomorphic to a Hilbert space''. 
\end{remark}

As one can easily see, a necessary and sufficient condition for $\pi\circ T|_M$ to be an isomorphism (not necessarily onto) from $M$  into $X^\ast/M^\bot$, is the existence of a positive constant $c$, such that whenever $x\in X$ and $x^\ast\in X^\ast$ are such that $\langle x^\ast,x\rangle=0$, we have that
\begin{equation}
\label{isom}
||Tx+x^\ast||\geq c||Tx||\,.
\end{equation}
In order to give some geometric intuition to condition (\ref{isom}) we recall the following definition.
\begin{definition}
Let $X$ be a normed space and $x\,,y\in X$. We say that $x$ is orthogonal, in the sense of Birkhoff-James, to $y$ if
$$||x+\lambda y||\geq ||x||, \text{ for all }\lambda\in \mathbb{R}.$$ 
\end{definition}
For more details about this notion of orthogonality the interested reader is referred to \cite{amir} and \cite{Istr}. 

It is easy to see that if whenever $x\in X$ and $x^\ast\in X^\ast$ are such that $\langle x^\ast,x\rangle=0$, we have that
\begin{equation}
\nonumber
\label{bir}
Tx\;\bot\; x^\ast\,, 
\end{equation}
in the sense of Birkhoff-James, then $T$ satisfies condition (\ref{isom}). 

As a matter of fact Birkhoff-James orthogonality is not the only orthogonality relation that can be used to guarantee the validity of condition (\ref{isom}). To see this we recall that in \cite{Partington}, J. R. Partington has introduced the concept of \textit{boundedness} for an orthogonality relation in an arbitrary normed space as follows.

\begin{definition}
An orthogonality relation $\bot$ in a normed linear space is bounded if there exists $c>0$ such that if $x\bot y$ then
\begin{equation}
\nonumber
\label{bound}
||\lambda x+ y||\geq c||x||, \text{ whenever } |\lambda|\geq c.
\end{equation}  
\end{definition}

Several well-known orthogonality relations (for example Birkhoff-James or Diminnie orthogonality, see \cite{diminnie} and \cite{Partington} for more details) are bounded. 
\begin{definition}
\label{homog}
An orthogonality relation $\bot$ in a normed linear space is homogeneous if 
\begin{equation}
\nonumber
x \bot y\;\text{ implies that }\;ax \bot by, \text{ for all }\;a,b\in\mathbb{R}.
\end{equation}
\end{definition}
\begin{remark}
\label{mil}
In \cite{Milicic} it was shown that if an orthogonality relation $\bot$ is homogeneous, then its boundedness is equivalent to the existence of $c>0$, such that $x\bot y$ implies 
\begin{equation}
\nonumber
\label{bound2}
||x+ y||\geq c||x||.
\end{equation}  
\end{remark}
Therefore if whenever $x\in X$ and $x^\ast\in X^\ast$ are such that $\langle x^\ast,x\rangle=0$, we have that
\begin{equation}
\nonumber
\label{general}
Tx\;\bot\; x^\ast, 
\end{equation}
for a homogeneous and bounded orthogonality relation $\bot$, then $T$ satisfies (\ref{isom}). 

To state our next result we need one more definition.
\begin{definition}
\label{nondeg}
An orthogonality relation $\bot$, in a normed linear space, is non-degenerate, if $x\bot x$ implies that $x=0$.
\end{definition}
We can now prove the following Hilbert space characterization.
\begin{theorem}
\label{orthogonal}
A real reflexive Banach space $X$ is isomorphic to a Hilbert space if and only if there exists an isomorphism
$$T:X\rightarrow X^\ast\,,$$
onto $X^\ast$, such that 
\begin{equation}
\label{orth}
Tx\;\bot\; x^\ast, \text{ whenever }  \langle x^\ast,x\rangle=0\,,
\end{equation}
for a non-degenerate, homogeneous and bounded orthogonality relation $\bot$ in $X^\ast$. 
\end{theorem}
\begin{proof}
The necessity is obvious. To prove the sufficiency of our claim we will use Proposition \ref{th1}. To this end let $M$ be any closed subspace of $X$. By (\ref{orth}) and the discussion above,  the operator $T$ satisfies condition (\ref{isom}) and hence $\pi\circ T|_M$ is an isomorphism. It remains to show that $\pi\circ T|_M$ is onto $X^\ast/M^\bot$. 

Since $(\pi\circ T)(M)$ is closed it is enough to show that it is a dense subspace of $X^\ast/M^\bot$. Assume the contrary i.e. 
$$(\pi\circ T)(M)\neq X^\ast/M^\bot\,.$$
Then by the Hahn-Banach Theorem there exists $0\neq f\in (X^\ast/M^\bot)^\ast$ such that 
$$f(Tx)=0\,, \text{ for all } x\in M\,.$$ 
Since $X$ is reflexive so is $M$ and hence it is isometrically isomorphic to $(X^\ast/M^\bot)^\ast$. Therefore there exists $x\in M$, such that $$\langle Tx,x\rangle=f(Tx)=0$$ 
and thus again by (\ref{orth}) we get that $Tx\bot Tx$. Using the non-degeneracy of $\bot$ and the injectivity of $T$ we get that $x$ and consequently $f$ have to be 0, which is a contradiction. 

Hence $\pi\circ T|_M$ is an isomorphism onto $X^\ast/M^\bot$ and by Proposition \ref{th1} the self-dual Banach space $X$ is isomorphic to a Hilbert space. 
\end{proof}
\section{The cosine of a linear operator revisited}
A simple situation where condition (\ref{isom}) holds is when there exists $c>0$, such that the operator $T$ satisfies 
\begin{equation}
\label{hes}
\langle Tx,x\rangle\geq c||Tx||^2,\text{ for all }x\in X\,.
\end{equation}
Recall the following well-known definition. 
\begin{definition}
Let $X$ be a real Banach space. We say that the linear operator $$T:D(T)\subseteq X\rightarrow X^\ast$$ 
\begin{itemize}
\item[(i)] is positive, if $\langle Tx,x\rangle\geq 0$, for all $x\in D(T)$.
\item[(ii)] is strictly positive, if $\langle Tx,x\rangle>0$, for all $x\in D(T)$, with $x\neq 0$.
\item[(iii)] is coercive, if there exists $c>0$, such that 
$$\langle Tx,x\rangle\geq c||x||^2\,,$$ 
for all $x\in D(T)$. 
\item[(iv)] is symmetric, if $\langle Tx,y\rangle=\langle Ty,x\rangle$, for all $x\,,y\in D(T)$.
\end{itemize}
\end{definition}
Note that since in all our previous considerations (in Section \ref{self}), the operator $T$ was an isomorphism inequality (\ref{hes}) would imply that the operator $T$ was actually coercive.
In the general case though, operators satisfying (\ref{hes}) form a much larger class than that of coercive operators. 
For example, see \cite{DY} and \cite{hess2} for more details, any positive, everywhere defined and symmetric operator $T$ satisfies (\ref{hes}).

On the other hand, unlike coercivity (see Remark \ref{remark}) condition (\ref{hes}) cannot guarantee on its own - i.e. when $T$ is no longer an isomorphism but just a continuous linear operator - the Hilbertian structure of $X$. Note that this is still the case even if $T$ has additional nice properties such as symmetry and positivity. It seems therefore quite natural that there may be some room between these two classes. To make things more precise we need the following definition.
\begin{definition}
\label{cosine}
Let $X$ be a real Banach space and let  $$T:D(T)\subseteq X\rightarrow X^\ast$$ be a positive linear operator. The cosine of $T$ is defined as follows:
\begin{equation}
\label{cos}
\cos T=\inf\left\{\frac{\langle Tx,x\rangle}{||Tx||\,||x||}\;,\;\text{ for all } 0\neq x\in D(T) \text{, such that } Tx\neq 0\right\}\,.
\end{equation}
\end{definition}

Using expression (\ref{cos}) one can define the angle $\phi(T)$ of the linear operator $T$, which has an obvious geometric interpretation: it measures the maximum turning effect of $T$. 

The above concepts were introduced, in the context of a complex Hilbert space, by K. Gustafson in \cite{gus1} and have attracted a lot of interest since then. We refer the interested reader to the book of K. Gustafson and D. Rao \cite{gus3}, for more details.

In order for the cosine of an operator to be a reliable tool, distinguishing between operators with different properties, it has to be positive for a large class of linear operators. As one can easily see this is the case for coercive everywhere defined - thus continuous - linear operators. 
On the other hand things fail dramatically for unbounded linear operators: it was shown by K. Gustafson and B. Zwahlen in \cite{gus2} and by P. Hess (in a somewhat more general context) in \cite{hess}, that the cosine of an unbounded linear operator is always 0.

To return to our main theme note that if $\cos T>0$, then $T$ satisfies (\ref{hes}). Thus non-coercive operators with positive cosine form the aforementioned  intermediate class, between (\ref{hes}) and coercivity. It turns out, as we shall see below, that if $X$ is not isomorphic to a Hilbert space then this class is quite small.

We need one more definition.
\begin{definition}[\cite{Partington}]
\label{properties}
An orthogonality relation $\bot$, in a normed linear space $X$ is 
\begin{itemize}
\item[(i)] symmetric, if $x\bot y$ implies $y\bot x$.
\item[(ii)] right additive, if $x\bot y$ and $x\bot z$ implies $x\bot (y+z)$.
\item[(iii)] resolvable, if for any $x\,,y$ there exists $a\in\mathbb{R}$, such that $x\bot (ax+y)$.
\item[(iv)] continuous, if $x_n\rightarrow x$, $y_n\rightarrow y$ and $x_n\bot y_n$, then $x\bot y$.
\end{itemize}
\end{definition}
It should be noted that an orthogonality relation having all six properties of Definitions \ref{homog}, \ref{nondeg} and \ref{properties} (i.e. except boundedness) exists in any separable Banach space (see Theorem 3 in \cite{Partington}).

If boundedness is added things change drastically as the following result of J. R. Partington \cite{Partington} illustrates.
\begin{theorem}[\cite{Partington}, Theorem 4]
\label{part}
If $X$ is a Banach space and $\bot$ is an orthogonality relation in $X$, that is non-degenerate, symmetric, homogeneous, right additive, resolvable, continuous and bounded, then $X$ is isomorphic to a Hilbert space.
\end{theorem}
In the sequel we identify $T^\ast$ with the restriction on $X$ of the adjoint of the linear operator $T:X\rightarrow X^\ast$ (which is defined on the whole of $X^{\ast\ast}$). 

Using Theorem \ref{part}, we can prove our main result for this section.
\begin{theorem}
\label{zero}
Let $X$ be a real Banach space, not isomorphic to a Hilbert space and
$$T:X\rightarrow X^\ast$$ 
a positive linear operator. If there exists $c>0$, such that 
\begin{equation}
\label{ineq}
||T^\ast x||\leq c||Tx||\,,\text{ for all }x\in X\,,
\end{equation}
then $\cos T=0$.
\end{theorem}
\begin{proof}
Assume the contrary and let $\cos T=\delta>0$. The linear operator
$$S:X\rightarrow X^\ast$$
defined by
$$S=\frac{1}{2}(T+T^\ast)\,.$$
is strictly positive, everywhere defined and hence continuous. We define the following orthogonality relation in $X$:
$$x\bot y\,,\text{ if } \langle Sx,y\rangle =0\,.$$
It is easy to see that $\bot$ is non-degenerate, symmetric, homogeneous, right additive, resolvable and continuous. To see that $\bot$  is also bounded take $x\bot y$ with $x\neq 0$. Then
\begin{eqnarray*}
||x+y||&=&\sup_{x^\ast\neq 0}\frac{\langle x^\ast,x+y\rangle}{||x^\ast||}\\
&\geq &\frac{\langle Sx,x+y\rangle}{||Sx||}
=\frac{\langle Tx,x\rangle}{||Sx||}\\
&\geq &\frac{2\langle Tx,x\rangle}{||Tx||(1+c)}\\
&\geq &\frac{2\delta}{1+c}||x||\,,
\end{eqnarray*}
where the second inequality is justified by (\ref{ineq}). 

Since $\bot$ is homogeneous, by Remark \ref{mil}, the orthogonality relation
$\bot$ is also bounded. 

Hence by Theorem \ref{part} the Banach space $X$ is isomorphic to a Hilbert space, which is a contradiction.
Thus $\cos T=0$.
\end{proof}
\begin{remark}
The class of operators satisfying (\ref{ineq}) is quite large as it includes positive, everywhere defined, symmetric linear operators.
\end{remark}
Combining this last remark with Theorem \ref{zero} we can have the following simple Hilbert space characterization.
\begin{corollary}
\label{symmetric}
A real Banach space $X$  is isomorphic to a Hilbert space if and only if there exists
a positive and symmetric linear operator
$$T:X\rightarrow X^\ast$$ 
with $\cos T>0$. 
\end{corollary}
It seems quite interesting that if $X$ is not isomorphic to a Hilbert space then an operator and its adjoint cannot have both positive cosines.
\begin{proposition}
Let $X$ be a real Banach space, not isomorphic to a Hilbert space and
$$A:X\rightarrow X^\ast$$ 
a positive linear operator with $\cos A>0$. Then $\cos A^\ast=0$.
\end{proposition}
\begin{proof}
Assume $\cos A=\delta>0$ and let $x\neq 0$. Then
\begin{eqnarray*}
||A^\ast x||&=&\sup_{y\neq 0}\frac{\langle A^\ast x,y\rangle}{||y||}\\
&\geq &\frac{\langle Ax,x\rangle}{||x||}\\
&\geq 
&\delta||Ax||\,.
\end{eqnarray*}
If $T=A^\ast$, then $T$ is a positive linear operator that satisfies (\ref{ineq}). Thus by Theorem \ref{zero} we get that $\cos A^\ast=0$.
\end{proof}
\subsection{An application to quadratic forms}
Recall that a continuous quadratic form on a normed space $X$ is a function
$$q:X\rightarrow \mathbb{R}$$
for which there exists a bounded bilinear form 
$$a:X\times X\rightarrow\mathbb{R}$$
such that 
$$q(x)=\frac{1}{2}a(x,x)\,.$$
It is well known (see for example \cite{Kalton}) that there exists a one-to-one correspondence between continuous quadratic forms and symmetric linear operators
$$T:X\rightarrow X^\ast$$
through the formula
\begin{equation}
\label{quadratic}
q(x)=\frac{1}{2}\langle Tx,x\rangle\,.
\end{equation}
Moreover, each continuous quadratic form is everywhere Frechet differentiable and its derivative is equal to $2T$, where $T$ is the symmetric operator in (\ref{quadratic}). 

Using Corollary \ref{symmetric}, we can have the following result.
\begin{proposition}
Let $X$ be a real Banach space, not isomorphic to a Hilbert space and
$$q:X\rightarrow \mathbb{R}$$
a continuous quadratic form. Then for any $\varepsilon>0$, there exists $x\in X$, such that
$$q(x)<\varepsilon ||q'(x))||\,||x||\,.$$
\end{proposition}
\begin{proof}
If $q(x)<0$, for some $x\in X$ we are done. If this is not the case then the symmetric linear operator $T$ that generates $q$, is positive and thus by Corollary \ref{symmetric} 
$$\cos T=0\,.$$ 
Hence for any $\varepsilon>0$, there exists $x\neq 0$, such that
$$\frac{\langle Tx,x\rangle}{||Tx||\,||x||}<\varepsilon\,.$$
Since $q(x)=\displaystyle\frac{1}{2}\langle Tx,x\rangle$ and $q'=2T$ the result follows.
\end{proof}
\subsection{An application to evolution triples}
We end this paper with an application of Theorem \ref{zero} to evolution triples.

Recall that we say that a Banach space is continuously and densely embedded into another Banach space $Y$, if there exists an injective, bounded linear operator 
$$i:X\rightarrow Y$$
such that $i(X)$ is dense in $Y$. We have the following Proposition.
\begin{proposition}
\label{evolution}
Let $X$ be a real reflexive Banach space that is continuously and densely embedded into a Hilbert space $H$ and assume that $X$ is not isomorphic to a Hilbert space. Then for any $\varepsilon>0$, there exists $x\in X$, such that
$$||i(x)||_H^2<\varepsilon ||i^\ast(i(x))||_{X^\ast}||x||_X\,,$$
where $i$ is the embedding operator from $X$ into $H$.
\end{proposition}
\begin{proof}
Since 
$$i:X\rightarrow H$$ 
is the embedding operator from $X$ into $H$ and the embedding is continuous and dense, then (after identifying $H$ with its dual space $H^\ast$) the embedding 
$$i^\ast:H\rightarrow X^\ast$$ 
is also continuous and dense (we say that $X$, $H$ and $X^\ast$ form an evolution or a Gelfand triple). 

Let 
$$T:X\rightarrow X^\ast$$
be defined by $T=i^\ast\circ i$. Then $T$ is a strictly positive, symmetric operator and by Corollary \ref{symmetric} we have that 
$$\cos T=0\,.$$
Thus for any $\varepsilon>0$, there exists $x\in X$, such that
$$\frac{\langle Tx,x\rangle}{||Tx||_{X^\ast}||x||_X}<\varepsilon.$$
But $\langle Tx,x\rangle=(i(x),i(x))_H=||i(x)||_H^2$ and hence the result follows.
\end{proof}
A concrete example of the above situation is the following:
\begin{example}
Let $\Omega\subseteq \mathbb{R}^N$, open and bounded and assume $p>2$. Then for every $\varepsilon>0$, there exists $f\in L^p(\Omega)$, such that
$$||f||_{L^2(\Omega)}^2<\varepsilon ||f||_{L^q(\Omega)}||f||_{L^p(\Omega)}\,,$$
where $\displaystyle\frac{1}{p}+\frac{1}{q}=1$.
\end{example}
\begin{proof}
Let $i:L^p(\Omega)\rightarrow L^2(\Omega)$ be the identity operator and use the previous proposition.
\end{proof}

\begin{acknowledgment}
The author would like to thank Dr. D. Drivaliaris and Mr. M. Garagai for many fruitful discussions. 
\end{acknowledgment}



\begin{thebibliography}{99}
\bibitem{amir} D. Amir, \textit{Characterizations of inner product spaces}, Oper. Theory Adv. Appl., vol. 20, Birkhäuser, Basel,
1986.
\bibitem{Brezis1} H. Brezis, \textit{\'{E}quations et in\'equations non lin\'eaires dans les espaces vectoriels en dualit\'e}, Ann. Inst. Fourier (Grenoble) \textbf{18} (1968), 115--175.
\bibitem{diminnie} C. R. Diminnie, \textit{A new orthogonality relation for normed linear spaces}, Math. Nachr. \textbf{114} (1983), 197--203.
\bibitem{DY} D. Drivaliaris, N. Yannakakis, \textit{Hilbert space structure and positive operators}, J. Math Anal. Appl. \textbf{305} (2005), no. 2, 560--565.
\bibitem{Ernst} E. Ernst, M. Th\' era, \textit{A converse to the Lions-Stampacchia theorem}, ESAIM Control Optim. Calc. Var. \textbf{15} (2009), no. 4, 810--817.
\bibitem{gus1} K. Gustafson, \textit{The angle of an operator and positive operator products}, Bull. Amer. Math. Soc., \textbf{74} (1968), 488--492. 
\bibitem{gus2} K. Gustafson, B. Zwahlen, \textit{On the cosine of unbounded operators}, Acta Sci. Math. \textbf{30} (1969), 33--34. 
\bibitem{gus3} K. Gustafson, D. Rao, \textit{Numerical Range. The field of values of linear operators and matrices}. Universitext. Springer-Verlag, New York, 1997.
\bibitem{hess} P. Hess, \textit{A remark on the cosine of linear operators}, Acta Sci. Math. \textbf{32} (1971), 267--269.
\bibitem{hess2} P. Hess, \textit{A remark on a class of linear monotone operators}, Math. Z. \textbf{125} (1972), 104--106. 
\bibitem{Istr} V. Istr{\u{a}}{\c{t}}escu, \textit{Inner product structures}. Theory and applications. Mathematics and its Applications, vol. 25, D. Reidel Publishing Co., Dordrecht, 1987.
\bibitem{Kalton} N. Kalton, S. Konyagin, L. Vesel{\'y}, \textit{Delta-semidefinite and delta-convex quadratic forms in Banach spaces}, Positivity \textbf{12} (2008), 221--240.
\bibitem{Lin} B. L. Lin, \textit{On Banach spaces isomorphic to its conjugate},
Studies and essays (presented to Yu-why Chen on his 60th birthday, April 1, 1970),
Math. Res. Center, Nat. Taiwan Univ., Taipei, (1970), 151--156.
\bibitem{LinTza} J. Lindenstrauss, L. Tzafriri, \textit{On the complemented subspaces problem},
Israel J. Math. \textbf{9} (1971), 263--269.
\bibitem{maugeri} A. Maugeri, F. Raciti, \textit{On existence theorems for monotone and nonmonotone variational inequalities}, J. Convex Anal. \textbf{16} (2009), no. 3-4, 899--911.
\bibitem{Milicic} P. M. Mili{\v{c}}i{\'c}, \textit{On isomorphisms by orthogonality of a normed space and an inner product space}, Publ. Inst. Math. (Beograd) (N.S.) \textbf{59(73)} (1996), 89--94.
\bibitem{Partington} J. R. Partington, \textit{Orthogonality in normed spaces}, Bull. Austral. Math. Soc. \textbf{33} (1986), no. 3, 449--455.
\bibitem{Stampacchia1} G. Stampacchia,  \textit{Formes bilin\'eaires coercitives sur les ensembles convexes}, C. R. Acad. Sci. Paris \textbf{258} (1964), 4413--4416.
\bibitem{Stampacchia2} J. L. Lions, G. Stampacchia,  \textit{Variational inequalities}, Comm. Pure Appl. Math. \textbf{20} (1967), 493--519. 

\end{thebibliography}
\end{document}